\documentclass[11pt]{article}
\usepackage{indentfirst}

\usepackage{float}
\usepackage{booktabs}
\usepackage{makecell}
\usepackage{epsfig}
\usepackage{url}
\usepackage{ulem}
\usepackage{amsmath}
\usepackage{amsfonts}
\usepackage{color}
\usepackage{amssymb}
\usepackage{graphicx,amsfonts,mathrsfs}
\usepackage{epstopdf}
\usepackage{enumerate}
\usepackage[all]{xy}
\usepackage{amsthm}
\usepackage{array}

\def\id{{\rm \textsf{id}}}
\def\Aut{{\rm \textsf{Aut}}}

\allowdisplaybreaks

\begin{document}
\begin{sloppypar}

\newtheorem{theorem}{Theorem}[section]
\newtheorem{problem}[theorem]{Problem}
\newtheorem{corollary}[theorem]{Corollary}
\newtheorem{definition}[theorem]{Definition}
\newtheorem{conjecture}[theorem]{Conjecture}
\newtheorem{question}[theorem]{Question}
\newtheorem{lemma}[theorem]{Lemma}
\newtheorem{proposition}[theorem]{Proposition}
\newtheorem{fact}[theorem]{Fact}
\newtheorem{observation}[theorem]{Observation}
\newtheorem{example}[theorem]{Example}
\newcommand{\remark}{\medskip\par\noindent {\bf Remark.~~}}
\newcommand{\pp}{{\it p.}}
\newcommand{\de}{\em}

\title{  {Generalized Cayley graphs of complete groups}\thanks{This research was supported by  NSFC (No. 12071484).
 E-mail addresses: liaoqianfen@163.com(Q. Liao),  wjliu6210@126.com(W. Liu, corresponding author).}}

\author{Qianfen Liao$^{a}$, Weijun Liu$^{b, c}$\\
{\small a. Department of Mathematics, Guangdong University of Education,} \\
{\small  Guangzhou,  Guangdong, 510303, P.R. China. }\\
{\small b. School of Mathematics and Statistics, Central South University,} \\
{\small  Changsha, Hunan, 410083, P.R. China. }\\
{\small c. College of General Education, Guangdong University of Science and Technology,}\\
{\small Dongguan Guangdong, 523083, P.R. China. }\\
}

\maketitle

\vspace{-0.5cm}

\begin{abstract}
A group $G$ is complete group if it satisfies $Z(G)=e$ and $Aut(G)=Inn(G)$.
In this paper, on the one hand, we study the basic properties of generalized Cayley graphs and characterize two classes isomorphic generalized generalized Cayley graphs of complete groups.
 On the other hand, we give the sufficient and necessary conditions of complete group to be $GCI$ group and restricted $GCI$ group.
As an application, we complete the classification of restricted $GCI$-groups for symmetric groups.
\end{abstract}

\section{Introduction}

The generalized Cayley graph, as a generalization of the Cayley graph, was first proposed in 1992 by D. Maru\v{s}i\v{c} et al. \cite{D.R.N}.
\begin{definition}\label{def1.1}
Given a finite group $G$, for $S\subseteq G$ and $\alpha \in \Aut(G)$, if they satisfy several conditions below:
\begin{enumerate}[{\rm(a)}]

\item $\alpha^2=\mathrm{id}$, where $\mathrm{id}$ is the identity of $\Aut(G)$;

\item for any $g \in G$, $(g^{-1})^\alpha g\notin S$;

\item for $g,h\in G$, if $(h^{-1})^\alpha g\in S$, then $(g^{-1})^\alpha h\in S$.
\end{enumerate}
\end{definition}
Then the graph with vertices $G$ and edges $\{\{g,h\}\mid (g^{-1})^\alpha h \in S\}$ is denoted by $GC(G,S,\alpha)$.
We call $S$ a generalized Cayley subset and $GC(G,S,\alpha)$ a generalized Cayley graph of $G$ with respect to the ordered pair $(S,\alpha)$.

Some basic properties about $GC(G,S,\alpha)$ deserve to be mentioned.
Firstly, $GC(G,S,\alpha)$ is a simple and undirected  graph.
Secondly, for each vertex $g \in G$, the set of vertices that adjacent to $g$ is $N(g)=\{g^\alpha s\mid s\in S\}$, so $GC(G,S,\alpha)$ is $|S|$-regular.
Thirdly, the subsets $\omega_\alpha=\{(g^{-1})^\alpha g\mid g\in G\}$, $\Omega_\alpha=\{g\mid g^\alpha=g^{-1} ~\text{and}~ g\notin \omega_\alpha\}$ and $\mho_\alpha=\{g\mid g^\alpha\neq g^{-1}\}=G\setminus \{\omega_\alpha\cup \Omega_\alpha\}$  construct a partition of $G$.
In addition, for generalized subset $S$, the condition $(b)$ implies that $S\cap \omega_\alpha=\emptyset$ and the condition (c) yields that $\alpha(S)=S^{-1}$.
Hence, if  $s\in S\cap \mho_\alpha$, then $\alpha(s^{-1})\in S$.
At last, the condition (a) above implies that $\alpha$ is $\id$ or an involutory automorphism of $G$.
Particularly, if $\alpha=\mathrm{id}$, $S$ is called a Cayley subset and it gives the Cayley graph.

As is well-known, Cayley graphs are vertex-transitive, but the generalized Cayley graphs may not be.
Maru\v{s}i\v{c} et al. provide specific examples in the article to illustrate this point.
In fact, vertex-transitive generalized Cayley graphs are rare.
Hujdurovi\'{c} et al. \cite{A.K.D} construct a series of non-Cayley vertex-transitive generalized Cayley graphs.
Subsequently, to investigate the isomorphism problem of generalized Cayley graphs, Yang, Liu and Feng \cite{Y2} introduce the definitions of $GCI$ groups and restricted $GCI$ groups.
Then $GCI$ groups and restricted $GCI$ groups for certain special groups have been studied, including cyclic groups, dihedral groups, alternating groups and non-abelian simple groups, see references \cite{L,Y1,Z}.

A group $G$ is complete group if it satisfies $Z(G)=e$ and $Aut(G)=Inn(G)$.
In this paper, we study the basic properties of generalized Cayley graphs of complete group.
As an application, we complete the classification of restricted $GCI$-groups for symmetric groups.

We introduce some notations in group theory.
For a group $G$ and element $g\in G$, let $C_G(g)=\{h\in G\mid hg=gh\}$ be the centralizer of $g$ in $G$, and $Z(G)=\{h\in G\mid gh=hg~\text{for all}~ g\in G\}=\cap_{g\in G}C_G(g)$ the central of $G$.
Given a subset $H$ of $G$, let $N_G(H)=\{g\in G\mid Hg=gH\}$ be the normalizer of $H$ in $G$.
Two elements $g$ and $h$ are conjugate in $G$ if there exists element $x\in G$ such that $g=xhx^{-1}$.
Let $C(g)$ be the conjugate class containing $G$, that is the set of elements that conjugate to $g$.
Then $|C(g)|=|G:C_G(g)|$.
The conjugate relation is an equivalent relation, and $G$ can be divided into some disjoint conjugate classes.
Let $[g, h]=g^{-1}h^{-1}gh$ be the commutator of elements $g$ and $h$.
For each element $g\in G$, let $\sigma(g): h\mapsto h^g=ghg^{-1}$ be the inner automorphism of $G$ induced by $g$.
Then $Inn(G)=\{\sigma(g)\mid g\in G\}$ is an subgroup of $Aut(G)$ and called the inner automorphism group of $G$.
Let  $\alpha$ and $\beta$ be permutations on $G$, for any $g\in G$, define $g^{\beta \alpha}=(g^\alpha)^\beta$.
In the following discussion of this paper, we assume that $G$ is complete group and all generalized Cayley graphs are induced by some involutory automorphism.

\section{Basic properties}

In this section, we study the properties of generalized Cayley graphs of complete graphs.
Firstly, some observations about the involutory automorphism of complete group $G$ are presented.

\begin{observation}
For $g\in G$,  $\sigma(g)$ is involutory automorphisms of $G$ if and only if
$g$ is involution of $G$.
\end{observation}

\begin{observation}\label{ob2.2}
The automorphisms $\sigma(g)$ and $\sigma(h)$ of $G$ are conjugate in $Aut(G)$ if and only if $g$ and $h$ are conjugate in $G$.
\end{observation}
\begin{proof}
If $\sigma(g)$ and $\sigma(h)$ are conjugate in $Aut(G)$, let $\sigma(y)$ be an automorphism of $G$ such that $\sigma(h)=\sigma(y)\sigma(g)(\sigma(y))^{-1}$.
Since $(\sigma(y))^{-1}=\sigma(y^{-1})$, we obtain that for any element $x\in G$,
\begin{equation*}
x^{\sigma(h)}=hxh^{-1}=x^{\sigma(y)\sigma(g)\sigma(y^{-1})}
=ygy^{-1}xyg^{-1}y^{-1}
=(ygy^{-1})x(gyy^{-1})^{-1}.
\end{equation*}
It follows that $h^{-1}(ygy^{-1})\in C_G(x)$.
From the arbitrary of $x$, we have $h^{-1}(ygy^{-1})\in Z(G)$.
Recall that $G$ is complete group, then $Z(G)=e$, which yields that $h=ygy^{-1}$.

Conversely, if $g$ and $h$ are conjugate in $G$, which means there exists element $x$ such that $g=xhx^{-1}$.
It can be verified that $\sigma(g)=\sigma(xhx^{-1})=\sigma(x)\sigma(h)\sigma(x)^{-1}$.
Thus, $\sigma(g)$ and $\sigma(h)$ are conjugate.
\end{proof}

Let $\sigma(g)$ be any involutory automorphism of $G$, then we find the relationship between  subset $\mho_{\sigma(g)}(G)$ and the order of element.
\begin{observation}\label{ob2.3}
$x\in \mho_{\sigma(g)}(G)$ if and only if $gx$ is not involution of $G$.
\end{observation}
\begin{proof}
$x\in \mho_{\sigma(g)}(G)$ is equivalent to $x^{\sigma(g)}=gxg\neq x^{-1}$.
Since $gxg\neq x^{-1}$ is equivalent to $xg\neq (xg)^{-1}$, the proof is complete.
\end{proof}

Let $K_{\sigma(g)}(G)=\omega_{\sigma(g)}(G)\cup \Omega_{\sigma(g)}(G)$.
The Observation \ref{ob2.3} also indicates that $x\in K_{\sigma(g)}(G)$ if and only if $gx$ is involution of $G$.

\begin{lemma}\label{fact2}
$g\in \Omega_{\sigma(g)}(G)$.
\end{lemma}
\begin{proof}
Since $g^{\sigma(g)}=g$, $g\in \omega_{\sigma(g)}(G)\cup \Omega_{\sigma(g)}(G)$.
If $g=[g,h]$ for some $h\in G$, then $g=g^{-1}h^{-1}gh$.
It follows that $gh=h$ and then $g=e$, which is a contradiction.
Thus, $g\in \Omega_{\sigma(g)}(G)$.
\end{proof}

For any element $x\in G$, $L_x: h\mapsto xh$ for each $h\in G$ is a permutation on $G$.
For any generalized Cayley graph $GC(G, S, \sigma(g))$, we find that there is a subgroup of $Aut(GC(G, S, \sigma(g)))$.
\begin{proposition}
$L(C_G(g))\leq Aut(GC(G, S, \sigma(g)))$.
\end{proposition}
\begin{proof}
For any element $x\in C_G(g)$ and $h_1, h_2\in G$,
\begin{align*}
((h_1^{L_x})^{-1})^{\sigma(g)}(h_2^{L_x})
&=((xh_1)^{-1})^{\sigma(g)}xh_2\\
&=gh_1^{-1}x^{-1}g^{-1}xh_2\\
&=gh_1^{-1}g^{-1}h_2=(h_1^{-1})^{\sigma(g)} h_2.
\end{align*}
Thus, $\{h_1^{L_x}, h_2^{L_x}\}\in E(GC(G, S, \sigma(g)))$ if and only if $\{h_1, h_2\}\in E(GC(G, S, \sigma(g)))$, which implies $L_x\in Aut(GC(G, S, \sigma(g)))$.
Now we give an explanation for $L(C_G(g))$ is a subgroup.
For any elements $x,y\in C_G(g)$ and $h\in G$, according to $h^{L_xL_y}=(h^{L_y})^{L_x}=xyh=h^{L_{xy}}$ and $xy\in C_G(g)$, it follows that $L_xL_y=L_{xy}\in L(C_G(g))$.
Furthermore, from $L(x)L(x^{-1})=L(xx^{-1})=id$, we have $L(x)^{-1}=L(x^{-1})\in L(C_G(g))$.
Therefore, $L(C_G(g))$ is a subgroup of $Aut(GC(G, S, \sigma(g)))$.
\end{proof}

For any involutory automorphism $\alpha$ of $G$, let $\mathrm{Fix}(\alpha)=\{h \in G | h^\alpha= h\}$ and we have the following lemma.

\begin{lemma}\label{lem0}
$|\omega_\alpha(G)|=\frac{|G|}{\mathrm{Fix}(\alpha)}$.
\end{lemma}

With the help of Lemma \ref{lem0}, we obtain the following result.

\begin{proposition}
The complete graph is not generalized Cayley graph of any complete group $G$.
\end{proposition}
\begin{proof}
For any involutory automorphism $\sigma(g)$ of $G$, note that \begin{equation*}
\mathrm{Fix}(\sigma(g))=\{h\in G\mid h^{\sigma(g)}=ghg^{-1}=h\}=\{h\in G\mid h^{-1}gh=g\}=C_G(g),
\end{equation*}
hence $|\omega_{\sigma(g)}(G)|=|G:C_G(g)|=|C(g)|$.
It is clear that $g\in C(g)$.
We claim that $|C(g)|\neq 1$.
Otherwise, $C(g)=g$, which means that for any element $x\in G$, $xgx^{-1}=g$.
It follows that $g\in Z(G)$, which contrary to $G$ is complete group.
Thus $|\omega_{\sigma(g)}(G)|=|C(g)|\geq 2$ and for any generalized Cayley subset $S$ of $G$, $|S|\leq |G|-2$ holds.
Therefore, any generalized Cayley graph of $G$ cannot be complete graph.
\end{proof}

\section{Several classes isomorphic generalized Cayley graphs}

In this section, we discuss the isomorphism problem of generalized Cayley graphs of complete groups.

Let $g$ be an involution of $G$ and $\sigma(g)$  an involutory automorphism of $G$.
Define $[g]=\{[g,h]\mid h\in G\}$ and call it the set of commutators induced by $g$.
There is relationship between subset $\omega_{\sigma(g)}$ of $G$ and $[g]$.

\begin{observation}
$\omega_{\sigma(g)}=\{(h^{-1})^{\sigma(g)}\mid h\in G\}=\{gh^{-1}g^{-1}h\mid h\in G\}=[g].$
\end{observation}

Let $\sigma(g)$ be an involutory automorphism and  $S$  a generalized Cayley subset of $G$ induced by $\sigma(g)$.
Then we obtain several classes isomorphisms between generalized Cayley graphs.

\begin{lemma}\label{lem1}
For any element $x\in \Omega_{\sigma(g)}(G)$ and commutator $[g,h]\in [g]$, $x^{-1}[g,h]x^{-1}=[g, hx^{-1}]$.
\end{lemma}
\begin{proof}
Note that $x\in \Omega_{\sigma(g)}(G)$ implies that $x^{\sigma(g)}=gxg^{-1}=x^{-1}$, hence $gx=x^{-1}g=x^{-1}g^{-1}$.
Then
\begin{equation*}
x^{-1}[g,h]x^{-1}=x^{-1}g^{-1}h^{-1}ghx^{-1}=gxh^{-1}ghx^{-1}=[g, hx^{-1}].
\end{equation*}
\end{proof}

\begin{theorem}
For any $x\in \Omega_{\sigma(g)}(G)$, $GC(G, S, \sigma(g))\cong GC(G, xSx, \sigma(g))$.
\end{theorem}
\begin{proof}
As mentioned in the proof of Lemma \ref{lem1}, $x\in \Omega_{\sigma(g)}(G)$ implies that $x^{\sigma(g)}=x^{-1}$.
Since $S^{\sigma(g)}=S^{-1}$, we obtain that
\begin{equation*}
(xSx)^{\sigma(g)}=x^{\sigma(g)} S^{\sigma(g)} x^{\sigma(g)}=x^{-1}S^{-1}x^{-1}=(xSx)^{-1}.
\end{equation*}
If $xSx\cap \omega_{\sigma(g)}(G)\neq \emptyset$, then there exists element $s\in S$ such that $xsx=[g,h]$ for some $h\in G$.
By Lemma \ref{lem1}, $s=x^{-1}[g,h]x^{-1}=[g, hx^{-1}]\in [g]$, which is contrary to $S\cap \omega_{\sigma(g)}(G)=\emptyset$.
Thus, $xSx$ is a generalized Cayley subset of $G$ induced by $\sigma(g)$.
Define the map $\varphi: y\mapsto yx$ for each $y\in G$.
Clearly, $\varphi$ is one to one mapping between vertices of $GC(G, S, \sigma(g))$ to $GC(G, xSx, \sigma(g))$.
For any vertices $y_1$ and $y_2$, since
\begin{equation*}
((y_1^\varphi)^{-1})^{\sigma(g)} y_2^\varphi=(x^{-1}y_1^{-1})^{\sigma(g)} y_2x=(x^{-1})^{\sigma(g)} (y_1^{-1})^{\sigma(g)} y_2x=xgy_1^{-1}gy_2x=x((y_1^{-1})^{\sigma(g)} y_2)x,
\end{equation*}
$\{y_1^\varphi, y_2^\varphi\}\in E(GC(G, xSx, \sigma(g)))$ if and only if $\{y_1, y_2\}\in E(GC(G, S, \sigma(g)))$.
Therefore, we conclude that $GC(G, S, \sigma(g))\cong GC(G, xSx, \sigma(g))$.
\end{proof}

\begin{theorem}
For any element $x\in N_G(S)$, $GC(G, S, \sigma(g))\cong GC(G, [g,x]S, \sigma(g))$.
\end{theorem}
\begin{proof}
Note that
\begin{equation*}
([g, x]S)^{\sigma(g)}=g(g^{-1}x^{-1}gxS)g^{-1}=x^{-1}gxSg=x^{-1}gSxg
\end{equation*}
 as $xS=Sx$.
Recall that $S^{\sigma(g)}=gSg=S^{-1}$, thus
\begin{equation*}
([g, x]S)^{-1}=(g^{-1}x^{-1}gxS)^{-1}=(g^{-1}x^{-1}gSx)^{-1}
=(g^{-1}x^{-1}S^{-1}gx)^{-1}=x^{-1}gSxg.
\end{equation*}
It indicates that $([g, x]S)^{\sigma(g)}=([g, x]S)^{-1}$.
If $[g, x]S\cap \omega_{\sigma(g)}(G)\neq \emptyset$, then there exists element $s\in S$ such that $[g, x]s=[g, b]$ for some $b\in G$.
This equality gives that $s=x^{-1}g^{-1}xb^{-1}gb$.
Since $xS=Sx$, there exists $s'\in S$ such that $s'=xsx^{-1}=g^{-1}xb^{-1}gbx^{-1}=[g, bx^{-1}]\in \omega_{\sigma(g)}(G)$, which is a contradiction.
Thus, $[g, x]S$ is a generalized Cayley subset of $G$ induced by $\sigma(g)$.
Now we define the map $\varphi:h\mapsto hx$ for each $h\in G$.
For any elements $h_1$ and $h_2$ in $G$, observe that
\begin{equation*}
((h_1^\varphi)^{-1})^{\sigma(g)} h_2^\varphi=gx^{-1}h_1^{-1}g^{-1}h_2x.
\end{equation*}
If $\{h_1, h_2\}\in E(GC(G, S, \sigma(g)))$, then $(h_1^{-1})^{\sigma(g)} h_2=gh_1^{-1}g^{-1}h_2\in S$, which is equivalent to $h_1^{-1}gh_2\in gS$.
Assume that $h_1^{-1}gh_2=gs$ and $sx=xs'$, where $s, s'\in S$.
Then
\begin{equation*}
((h_1^\varphi)^{-1})^{\sigma(g)} h_2^\varphi=gx^{-1}g^{-1}sx=g^{-1}x^{-1}gxs'=[g,x]s'\in [g,x]S.
\end{equation*}
It follows that $\{h_1^\varphi, h_2^\varphi)\in E(GC(G, [g,x]S, \sigma(g)))$.
Conversely, if $\{h_1^\varphi, h_2^\varphi)\in E(GC(G, [g,x]S, \sigma(g))$,
then $((h_1^\varphi)^{-1})^{\sigma(g)} h_2^\varphi=gx^{-1}h_1^{-1}gh_2x\in [g,x]S$.
Combining the fact
\begin{equation*}
[g,x]S=g^{-1}x^{-1}gxS=gx^{-1}gSx,
\end{equation*}
we have $h_1^{-1}gh_2\in gS$, which yields that $gh_1^{-1}gh_2=(h_1^{-1})^{\sigma(g)} h_2\in S$.
Therefore, $\{h_1, h_2\}\in E(GC(G, S, \sigma(g)))$ if and only if $\{h_1^\varphi, h_2^\varphi)\in E(GC(G, [g,x]S, \sigma(g)))$.
\end{proof}

\section{Complete group and (restricted) GCI-group}

Since the complete group $G$ of odd order does not contain involution, it does not have involutory automorphism and generalized Cayley graphs.
In this section, we only consider the complete group of even order.

The definitions of restricted GCI-group and GCI-group are presented first.
\begin{definition}
For any two generalized Cayley graphs $X_i=GC(G,S_i,\alpha_i)$ $(i=1,2)$ of $G$ with $| S_i|\leq m$,
\begin{enumerate}[{\rm(a)}]

\item if $\alpha_1=\alpha_2=id$, $X_1\cong X_2$ implies there exists automorphism $\gamma$ such that $S_2=S_1^\gamma$, then we call $G$  a $m$-$CI$-group;

\item if $X_1\cong X_2$ implies $\alpha_2=\alpha_1^\gamma=\gamma \alpha_1 \gamma^{-1}$ and $S_2=g^{\alpha_2}S_1^\gamma g^{-1}$ for some $g \in G$ and automorphism $\gamma$, then we call $G$  an $m$-$GCI$-group;

\item if both $\alpha_1$ and $\alpha_2$ are involutions, $X_1\cong X_2$ implies $\alpha_2=\alpha_1^\gamma=\gamma \alpha_1 \gamma^{-1}$ and $S_2=g^{\alpha_2}S_1^\gamma g^{-1}$ for some $g \in G$ and automorphism $\gamma$, then we call $G$ a restricted $m$-$GCI$-group.
\end{enumerate}

In particular, if $m=|G|$, we simply call $G$ a $CI$-group, $GCI$-group or a restricted $GCI$-group, respectively.
The graph isomorphism among $\rm{(a)}$ is called $CI$ isomorphism.
The graph isomorphisms among $\rm{(b)}$ and $\rm{(c)}$ are called $GCI$ isomorphisms.
\end{definition}

Regard to the $GCI$ isomorphisms of generalized Cayley graphs, the following lemma holds.

\begin{lemma}\label{lem4.0}\cite{Y1}
The $GCI$ isomorphism relation is an equivalence relation.
\end{lemma}

\begin{lemma}\label{lem4.1}\cite{A.K.P.A}
$GC(G, S, \alpha)\cong GC(G, S^\beta, \alpha^\beta)$ for any $\beta\in Aut(G)$, where $\alpha^\beta=\beta^{-1} \alpha \beta$.
\end{lemma}

Observe that if $S$ is a generalized Cayley subset of $G$ induced by the involutory automorphism $\sigma(g)$ and it satisfies $g\notin S$, then $e\notin S$ and $S^{\sigma(g)}=gSg=S^{-1}$.
It follows that $gS=S^{-1}g^{-1}=(gS)^{-1}$, and then $gS$ is a Cayley subset of $G$.
Given generalized Cayley graphs $GC(G, S_1, \sigma(g_1))$ and $GC(G, S_2, \sigma(g_2))$.
It is interesting to see the relationship between the $GCI$ isomorphism of generalized Cayley graphs and the isomorphism of Cayley graphs.

Let $S_1$ and $S_2$ be generalized Cayley subsets of $G$ induced by involutory automorphisms $\sigma(g_1)$ and $\sigma(g_2)$ satisfying $g_1\notin S_1$ and $g_2\notin S_2$, respectively.
Then the following conclusion holds.
\begin{proposition}\label{pro4.4}
If the generalized Cayley graph $GC(G, S_1, \sigma(g_1))$ is $GCI$ isomorphic to $GC(G, S_2, \sigma(g_2))$, then $Cay(G, g_1S_1)$ is $CI$ isomorphic to $Cay(G, g_2S_2)$.
\end{proposition}
\begin{proof}
Since $GC(G, S_1, \sigma(g_1))$ is $GCI$ isomorphic to $GC(G, S_2, \sigma(g_2))$, there exist elements $h$ and $x$ such that $\sigma(g_2)=\sigma(g_1)^{\sigma(h)}=\sigma(hg_1h^{-1})$ and
\begin{equation}\label{eq4.4.1}
S_2=x^{\sigma(g_2)}S_1^{\sigma(h)}x^{-1}=g_2xg_2^{-1}hS_1h^{-1}x^{-1}.
\end{equation}
Observe that  $\sigma(g_2)=\sigma(hg_1h^{-1})$ implies that for any element $y\in G$,
\begin{equation}\label{eq4.4.2}
y^{\sigma(g_2)}=g_2yg_2^{-1}=y^{\sigma(hg_1h^{-1})}=(hg_1h^{-1})y (hg_1h^{-1})^{-1}.
\end{equation}
It follows that $g_2^{-1}(hg_1h^{-1})\in Z(G)$.
Since $Z(G)=e$, we have $g_2=hg_1h^{-1}$.
Combining Equalities (\ref{eq4.4.1}) and (\ref{eq4.4.2}), we obtain
\begin{equation*}
g_2S_2=xg_2^{-1}hS_1h^{-1}x^{-1}=xhg_1^{-1}S_1h^{-1}x^{-1}=xh(g_1S_1)(xh)^{-1}
=(g_1S_1)^{\sigma_{xh}}.
\end{equation*}
For any vertices $y_1$ and $y_2$, $\{y_1, y_2\}\in E(Cay(G, g_1S_1))$ if and only if $y_1^{-1}y_2\in g_1S_1$.
Since $\sigma_{xh}\in Aut(G)$, $y_1^{-1}y_2\in g_1S_1$ if and only if \begin{equation*}
(y_1^{-1})^{\sigma_{xh}} y_2^{\sigma_{xh}}\in (g_1S_1)^{\sigma_{xh}}=g_2S_2.
\end{equation*}
Thus $Cay(G, g_1S_1)\cong Cay(G, g_2S_2)$ under the inner isomorphism $\sigma_{xh}$.
\end{proof}

For each generalized Cayley subset $S$ induced by $\sigma(g)$, $gS$ is a Cayley subset.
However, the converse may not true, since for any generalized Cayley subset $S$ and involution $g$ of $G$, $gS$ may not be generalized Cayley subset.
In some special case, the converse implication of Proposition \ref{pro4.4} holds.

Let $g_1,g_2$ be two conjugate involutions of $G$.
Let $S_1$ and $S_2$ be Cayley subsets satisfying $S_1\cap C(g_1)=\emptyset$ and $S_2\cap C(g_2)=\emptyset$. Then the following result holds.

\begin{proposition}\label{pro4.5}
If $Cay(G, S_1)$ is $CI$ isomorphic to $Cay(G, S_2)$, then the generalized Cayley graph $GC(G, g_1S_1, \sigma(g_1))$ is $GCI$ isomorphic to $GC(G, g_2S_2, \sigma(g_2))$.
\end{proposition}
\begin{proof}
Firstly, we give an explanation for $g_1S_1$ and $g_2S_2$ are generalized Cayley subsets induced by $\sigma(g_1)$ and $\sigma(g_2)$ respectively.
Note that $(g_1S_1)^{\sigma(g_1)}=S_1g_1=S_1^{-1}g_1^{-1}=(g_1S_1)^{-1}$.
Since
\begin{equation*}
g_1S_1\cap \omega_{\sigma(g_1)}(G)=g_1(S_1\cap \{x^{-1}g_1x\mid x\in G\})=g_1(S_1\cap C(g_1))=\emptyset,
\end{equation*}
Thus, $g_1S_1$ is generalized Cayley subsets induced by $\sigma(g_1)$.
With the similar discussion, it can be verified that $g_2S_2$ is generalized Cayley subsets induced by $\sigma(g_2)$.
Assume that $Cay(G, S_1)$ is $CI$ isomorphic to $Cay(G, S_2)$ and $g_2=gg_1g^{-1}$ for some $g\in G$.
Then there exists automorphism $\sigma(h)$ such that $S_2=S_1^{\sigma(h)}$.
Let $x=hg_1g^{-1}g_2$.
Then
\begin{equation*}
S_2=hS_1h^{-1}=xg_2^{-1}gg_1^{-1}S_1g_1g^{-1}g_2x^{-1}.
\end{equation*}
It follows that
\begin{align*}
g_2S_2&=g_2xg_2^{-1}gg_1^{-1}S_1g_1g^{-1}g_2x^{-1}\\
&=x^{\sigma(g_2)}g(g_1S_1)g^{-1}g_2x^{-1}\\
&=x^{\sigma(g_2)}(g_1S_1)^{\sigma(g)}x^{-1}
\end{align*}
Thus, $GC(G, g_1S_1, \sigma(g_1))$ is $GCI$ isomorphic to $GC(G, g_2S_2, \sigma(g_2))$.
\end{proof}

A corollary follows from Propositions \ref{pro4.4} and \ref{pro4.5}.

\begin{corollary}
Let $g_1,g_2$ be two conjugate involutions of $G$.
Let $S_1$ and $S_2$ be Cayley subsets of $G$ satisfying $S_1\cap C(g_1)=\emptyset$ and $S_2\cap C(g_2)=\emptyset$.
Then $Cay(G, S_1)$ is $CI$ isomorphic to $Cay(G, S_2)$ if and only if $GC(G, g_1S_1, \sigma(g_1))$ is $GCI$ isomorphic to $GC(G, g_2S_2, \sigma(g_2))$.
\end{corollary}

Next we study the restricted GCI-group and GCI-group of complete groups.
An immediate consequence of Lemma \ref{fact2} is as following.

\begin{theorem}\label{thm4.1}
Any complete group $G$ of even order is not $GCI$-group.
\end{theorem}
\begin{proof}
Since the order of $G$ is even, it must contain an involution $g$ and then we obtain the involutory automorphism $\sigma(g)$ of $G$.
By Lemma \ref{fact2}, $g\in \Omega_{\sigma(g)}(G)$.
Then we have $GC(G, \{g\}, \sigma(g))\cong Cay(G, \{g\})\cong \frac{|G|}{2}K_2$.
But $\sigma(g)$ is not conjugate to the identity map, so $G$ is not $GCI$ group.
\end{proof}

Let $G_2$ be the subset  containing all involutions of $G$.
Then we obtain an equivalent condition for $G$ to be restricted $GCI$ group.
\begin{theorem}\label{thm4.2}
Let $g$ be an involution of $G$.
Then $G$  is restricted $GCI$  group if and only if $G_2=C(g)$ and for any generalized Cayley subsets $S_1$ and $S_2$ induced by $\sigma(g)$,
$GC(G, S_1, \sigma(g))\cong GC(G, S_2, \sigma(g))$ implies that there exists $\sigma(x)\in Aut(G)$ such that $gS_2=(gS_1)^{\sigma(x)}$.
\end{theorem}
\begin{proof}
Assume that $G$ is restricted $GCI$-group.
For any involution $h$ of $G$, according to the proof Theorem \ref{thm4.1}, we obtain that $GC(G, \{g\}, \sigma(g))\cong GC(G, \{h\}, \sigma(h))\cong \frac{|G|}{2}K_2$.
From the definition of restricted $GCI$-group, it follows that $\sigma(g)$ and $\sigma(h)$ are conjugate in $Aut(G)$.
By Observation \ref{ob2.2}, $g$ and $h$ are conjugate in $G$.
Thus, $G_2=C(g)$.
If $GC(G, S_1, \sigma(g))\cong GC(G, S_2, \sigma(g))$, then from the definition of restricted $GCI$ group, we have $\sigma(g)=\sigma(g)^{\sigma(x)}=\sigma(xgx^{-1})$ and $S_2=y^{\sigma(g)}S_1^{\sigma(x)} y^{-1}$ for some elements $x,y\in G$.
The latter equality is equivalent to $gS_2=ygS_1^{\sigma(x)} y^{-1}$.
As the proof Proposition \ref{pro4.4},  $Z(G)=e$ and $\sigma(g)=\sigma(g)^{\sigma(x)}=\sigma(xgx^{-1})$ implies that $g=xgx^{-1}=g^{\sigma(x)}$.
It follows that
\begin{equation*}
gS_2=y(gS_1)^{\sigma(x)} y^{-1}=(gS_1)^{\sigma(y)\sigma(x)}=(gS_1)^{\sigma(yx)}.
\end{equation*}

Conversely, $G_2=C(g)$ means that each involution of $G$ is conjugate to $g$.
According to Lemma \ref{lem4.1}, it follows that each generalized Cayley graph is $GCI$ isomorphic to a generalized Cayley graph induced by $\sigma(g)$.
Moreover,  Lemma \ref{lem4.1} tells us the $GCI$ isomorphism relation is an equivalence relation.
Thus, to show that $G$  is restricted $GCI$ complete group, it suffices to prove that the isomorphism between the generalized Cayley graphs induced by
$\sigma(g)$ is $GCI$ isomorphism.
For any generalized Cayley graphs $GC(G, S_1, \sigma(g))$ and $GC(G, S_2, \sigma(g))$, assume that $GC(G, S_1, \sigma(g))\cong GC(G, S_2, \sigma(g))$.
Then  there exists $\sigma(x)\in Aut(G)$ such that $gS_2=(gS_1)^{\sigma(x)}$.
Let $y=xg$.
Thus we infer that
\begin{align*}
S_2&=g(gS_1)^{\sigma(x)}
=g(gS_1)^{\sigma(y)\sigma(g)}
=g((gS_1)^{\sigma(g)})^{\sigma(y)}\\
&=g(S_1g)^{\sigma(y)}
=gyS_1gy^{-1}
=gxgS_1gy^{_1}
=gxS_1^{\sigma(g)}y^{-1}
\end{align*}
Observe that $y^{\sigma(g)}=gxgg=gx$.
It follows that $S_2=y^{\sigma(g)}S_1^{\sigma(g)}y^{-1}$.
Clearly, $\sigma(g)^{\sigma(g)}=\sigma(g)$.
Therefore, $G$  is restricted $GCI$  group.
\end{proof}


In Lemma \ref{fact2}, we state that for any involutory automorphism $\sigma(g)$ of $G$, $g\in \Omega_{\sigma(g)}(G)$.
Especially, if $G$ is restricted $GCI$-group, then we further obtain the following result.

\begin{proposition}\label{pro4.6}
If $G$ is restricted $GCI$ complete group of even order, then $\Omega_{\sigma(g)}(G)=\{g\}$.
\end{proposition}
\begin{proof}
If there exists element $h\in \Omega_{\sigma(g)}(G)$ and $h\neq g$, then from Observation \ref{ob2.3}, $gh$ is involution of $G$.
Since $G$ is restricted $GCI$ group, by Theorem \ref{thm4.2}, $gh$ and $g$ are conjugate.
It follows that there exists element $x\in G$ such that $gh=x^{-1}gx$, which implies $xghg=gxg=x^{\sigma(g)}$.
Then \begin{equation*}
(x^{-1})^{\sigma(g)}x=g^{-1}h^{-1}gx^{-1}x=g^{-1}h^{-1}g=(h^{-1})^{\sigma(g)}=h.
\end{equation*}
Thus, $h\in  \Omega_{\sigma(g)}(G)\cap \omega_{\sigma(g)}(G)$, which is a contradiction.
Therefore, $\Omega_{\sigma(g)}(G)=\{g\}$.
\end{proof}

If $G$ is complete group but $4\nmid |G|$, we have the same result as Proposition \ref{pro4.6}.

\begin{proposition}
If $G$ is complete group of even order and $4\nmid|G|$, then for any involutory automorphism $\sigma(g)$ of $G$, $\Omega_{\sigma(g)}(G)=\{g\}$.
\end{proposition}
\begin{proof}
Assume that there exists element $h\in \Omega_{\sigma(g)}(G)$ and $h\neq g$.
From Lemma \ref{fact2}, we know that $\{g, h\}$ is a generalized Cayley subset.
For any vertex $x$ in  $GC(G, \{g, h\}, \sigma(g))$, $x\rightarrow gx\rightarrow xgh\rightarrow gxgh$ is a cycle of length $4$.
Thus, $GC(G, \{g, h\}, \sigma(g))\cong \frac{|G|}{4}$, which is contrary to $4\nmid |G|$.
Therefore, $\Omega_{\sigma(g)}(G)=\{g\}$.
\end{proof}

Next is  a application of Theorem \ref{thm4.2}.
It is known that the symmetric group $S_n(n\neq6)$ is complete group.
There are two basic facts about symmetric groups:

(1) Without the consideration of order, each permutation of $S_n$ can be decomposed  into product of some disjoint cycle;

(2) two elements of $S_n$ are conjugate if and only if their cycle form are the same.

For example, $(12)$ and $(34)$ are elements of $S_4$, and they are conjugate since both them have a cycle of length $2$.
But $(12)$ is not conjugate to  $(12)(34)$.

As an application of Theorem \ref{thm4.2}, we determined the restricted $GCI$ group among symmetric groups $S_n$.

\begin{corollary}\label{cor4.5}
The symmetric group $S_n(n\neq 6)$ is restricted $GCI$-group if and only if $n=3$.
\end{corollary}
\begin{proof}
For $n\geq 4$ and $n\neq 6$, $(12)$ and $(12)(34)$ are involutions of the symmetric group $S_n$.
Since the cycle form of $(12)$ and $(12)(34)$ are different, $(12)$ and $(12)(34)$ are not conjugate.
By Theorem \ref{thm4.2}, $S_n(n\geq 4~\text{and}~n\neq 6)$ is not restricted $GCI$-group.
Now it suffices to prove that $S_3$ is restricted $GCI$-group.
Observe that $S_3=\{(1), (12), (13), (23), (123), (132)\}$ and all involutions of it are conjugate.
Choose the involution $\sigma_{(12)}$.
Then we have $\omega_{(12)}(S_3)=\{(1), (123), (132)\}$, $\Omega_{(12)}(S_3)=\{(12)\}$ and $\mho_{(12)}(S_3)=\{(23), (13)\}$.
For any positive integer $d\leq 3$, the $d$-valent generalized Cayley graph of $G$ induced by $\sigma_{(12)}$ is unique.
By Theorem \ref{thm4.2}, $S_3$ is restricted $GCI$-group.
\end{proof}

In reference \cite{8}, the author depict the automorphism of $S_6$ in detail.
Note that $S_6$ can be generated by $A= \{(12),(13),(14),(15),(16)\}$.
Define a map $\phi$ on $S_6$ as follows:
\begin{align*}
&(12)^{\phi} = (12)(36)(45), (13)^{\phi} = (16)(24)(35), (14)^{\phi} = (13)(25)(46),\\
&(15)^{\phi} = (15)(26)(34), (16)^{\phi} = (14)(23)(56).
\end{align*}
It can be verified that $\phi$ is outer automorphism of $S_6$.
Let~$x=(12345)$, $\sigma(x)$ is the inner automorphism induced by $x$.
Let $\delta=\sigma(x)\phi$, then $\delta$ is an involutory outer automorphism of $S_6$.

\begin{lemma}\cite{8}\label{lem4.4}
For any element $h\in S_6$ , define the map $\delta_g : h^{\delta_g}= gh^\delta g^{-1}=(h^\phi)^{\sigma(gx)}$.
Then $\mathrm{Aut}(S_6)=\{\sigma(g)\mid g \in S_6\}\cup \{\delta_g \mid g\in S_6\}$.
\end{lemma}

Using Lemma \ref{lem4.4}, we prove the next proposition.

\begin{proposition}\label{pro4.6}
$S_6$ is not restricted $\mathrm{GCI}$ group.
\end{proposition}

\begin{proof} 
Note that both $\sigma(12)$ and $\sigma\left((12)(34)\right)$ are involutory automorphisms of $S_6$, and
\begin{equation*}
\mathrm{GC}(S_6, \{(12)\}, \sigma(12))\cong \mathrm{GC}(S_6, \{(12)(34)\}, \sigma\left((12)(34))\right)\cong 360 K_2.
\end{equation*}
Now we aim to prove that $\sigma(12)$ and $\sigma\left((12)(34)\right)$ are not conjugate in $\mathrm{Aut(S_6)}$.

Firstly, for any element $g\in S_6$, $\sigma(12)^{\sigma(g)}=\sigma(g)\sigma(12) \left(\sigma(g)\right)^{-1}=\sigma\left(g(12)g^{-1}\right)$.
If $\sigma\left(g(12)g^{-1}\right)=\sigma\left((12)(34)\right)$, then for any element  $h\in S_6$, 
\begin{equation*}
h^{\sigma\left(g(12)g^{-1}\right)}=g(12)g^{-1}hg(12)g^{-1}
=(12)(34)h(12)(34)=h^{\sigma\left((12)(34)\right)}.
\end{equation*}
It indicates that $(12)(34)g(12)g^{-1}\in Z(S_6)=e$, and then $(12)(34)=g(12)g^{-1}$.
But this equality cannot hold, as $(12)$ is not conjugate to $(12)(34)$ in $S_6$.
Thus, for any element $g\in S_6$, $\sigma(12)^{\sigma(g)}\neq \sigma\left((12)(34)\right)$.

We consider the conjugation map for $\sigma(12)$ under the automorphism in $\{\delta_g \mid g\in S_6\}$.
For any automorphism $\delta_g$, let $\gamma=\delta_{(g^{-1})^\delta}$.
Since $\delta$ is involutory automorphism, for any element $h\in G$, 
\begin{equation*}
h^{\gamma\delta_g}=(h^{\delta_g})^{\gamma}=(gh^{\delta}g^{-1})^{\gamma}
=(g^{-1})^\delta(gh^{\delta}g^{-1})^\delta g^\delta=(g^{-1})^\delta g^\delta h(g^{-1})^\delta g^\delta=h.
\end{equation*}
Then we deduce that $(\delta_g)^{-1}=\delta_{(g^{-1})^\delta}$.
It follows that 
\begin{align*}
h^{\delta_g \sigma(12) (\delta_g)^{-1}}&=h^{\delta_g \sigma(12) \delta_{(g^{-1})^\delta}}\\
&=g\left((12)(g^{-1})^\delta h^\delta g^\delta (12)\right)^\delta g^{-1}\\
&=g(12)^\delta g^{-1} h g (12)^\delta g^{-1}\\
&=h^{\sigma(g(12)^\delta g^{-1})}.
\end{align*}
Since $(12)^\delta=(12)(36)(45)$, and $(12)(36)(45)$ is not conjugate to $(12)(34)$ in $S_6$, $g(12)^\delta g^{-1}=\left((12)(36)(45)\right)^{\sigma(g)}\neq (12)(34)$.
Thus, $\delta_g \sigma(12) (\delta_g)^{-1}\neq \sigma\left((12)(34)\right)$.

In conclusion, $\sigma(12)$ is not conjugate to $\sigma\left((12)(34)\right)$ in $\mathrm{Aut(S_6)}$. 
Therefore, $S_6$ is not restricted $\mathrm{GCI}$.
\end{proof}

Combining Corollary \ref{cor4.5} and Proposition \ref{pro4.6}, we give the characterization of restricted $\mathrm{GCI}$ groups for symmetric groups.

\begin{theorem}
The symmetric group $S_n$ is restricted $\mathrm{GCI}$ group if and only if $n=3$.
\end{theorem}

To determine the $GCI$ groups in symmetric groups completely, we only need to consider $S_6$.

\end{sloppypar}


\frenchspacing
\begin{thebibliography}{99}

\bibitem{Bondy}
J. A. Bondy, U. S. R. Murty, {Graph theory with applications}, Macmillan London and Elsevier, New York, (1976).

\bibitem{A.K.D}
A. Hujdurovi\'{c}, K. Kutnar, D. Maru\v{s}i\v{c}, {Vertex-transitive generalized Cayley graphs which are not Cayley graphs}, European J. Combin.  {46} (2015), 45--50.

\bibitem{A.K.P.A}
A. Hujdurovi\'{c}, K. Kutnar, P. Petecki, A. Tanana,  {On automorphisms and structural properties of generalized Cayley graphs},  Filomat 31 (2017), 4033--4040.


\bibitem{L}
Qianfen Liao, Weijun Liu, {GCI-property of some groups}, Appl. Math. Comput.  438 (2023).

\bibitem{D.R.N}
D. Maru\v{s}i\v{c}, R. Scapellato, N. Zagaglia Salvi, {Generalized Cayley graphs}, Discrete Math.  102 (1992), 279--285.

\bibitem{Y1}
X. Yang, W.J. Liu, J. Chen, L.H. Feng, {GCI-groups in the alternating groups}, Appl. Math. Comput.  303 (2017), 42--47.

\bibitem{Y2}
X. Yang, W.J. Liu, L.H. Feng, {Isomorphisms of generalized Cayley graphs}, Ars. Math. Contemp.  {15} (2018), 407--424.

\bibitem{Z}
X.M. Zhu, W.J. Liu, X. Yang, {The isomorphism of generalized Cayley graphs on finite non-abelian simple groups}, Discrete Math. 346 (2023), 1--10.

\bibitem{8}
Wu M Y, Tang Q L, Luo X H, Jiang H L. A class of Hopf algebras on permutation groups $S_6$. J Mathematics(in Chinese), 2018, 38(5): 921--932.

\end{thebibliography}
\end{document}